\documentclass[3p]{elsarticle}
\usepackage{amssymb, amsmath, amsfonts, amsthm, graphics, epsfig, hyperref, fullpage}

\journal{---}
\hypersetup{
    bookmarks=true,         % show bookmarks bar?
    unicode=false,          % non-Latin characters in Acrobats bookmarks
    pdftoolbar=true,        % show Acrobats toolbar?
    pdfmenubar=true,        % show Acrobats menu?
    pdffitwindow=true,      % page fit to window when opened
    pdfnewwindow=true,      % links in new window
    colorlinks=true,       % false: boxed links; true: colored links
    linkcolor=BrickRed,          % color of internal links - red
    citecolor=Emerald,        % color of links to bibliography  green
    filecolor=magenta,      % color of file links
    urlcolor=cyan           % color of external links
}

\newtheorem{thm}{Theorem}[section]

\newtheorem{lem}[thm]{Lemma}
\newtheorem{prop}[thm]{Proposition}

\newdefinition{defn}[thm]{Definition}
\newdefinition{exmp}[thm]{Example}

\newdefinition{prob}{Problem}
\newdefinition{rem}{Remark}
\newdefinition{ques}{Question}

\newcommand{\diam}{{\rm{diam}}}

\newcommand{\ceil}[1]{\ensuremath{\left\lceil #1\right\rceil} }
\newcommand{\size}[1]{\left \vert #1 \right \vert}

\DeclareMathOperator{\acop}{c_a}
\DeclareMathOperator{\cop}{c}

\def\cart{\, \Box \,}

\begin{document}

\begin{frontmatter}

\title{Fully Active Cops and Robbers}

\author[McGill]{Ilya Gromovikov}
\ead{ilya.gromovikov@mail.mcgill.ca} 
\author[URI]{William B. Kinnersley}
\ead{billk@uri.edu}
\author[DC,UM]{Ben Seamone}
\ead{bseamone@dawsoncollege.qc.ca}
   \address[McGill]{Department of History and Classical Studies, McGill University, Montreal, QC}
\address[URI]{
   Department of Mathematics, 
   University of Rhode Island, Kingston, RI}
\address[DC]{
   Mathematics Department, Dawson College, Montreal, QC}
\address[UM]{
   D\'epartement d'informatique et de recherche operationnelle, Universit\'e de Montreal, Montreal, QC}

\begin{abstract}
We study a variation of the classical pursuit-evasion game of Cops and Robbers in which agents are required to move to an adjacent vertex on every turn.  We explore how the minimum number of cops needed to catch the robber can change when this condition is added to the rules of the game.  We study this ``Fully Active Cops and Robbers'' game for a number of classes of graphs and present some open problems for future research.
\end{abstract}

\end{frontmatter}

\section{Introduction}

The game of Cops and Robbers played on graphs was introduced independently by Quillot \cite{Q83} and Nowakowski and Winkler \cite{NW83}.  The game is played between a set of pursuers (cops) and an evader (robber) who move from vertex to adjacent vertex in a graph.  The cops win if at least one cop is able to occupy the same vertex as the robber; the robber wins if he can avoid capture indefinitely.  In the original version of the game, the game begins with the cops choosing their starting vertices, followed by the robber choosing his; multiple cops may occupy the same vertex simultaneously.  The cops move first, with each cop either moving to a vertex adjacent to her current position or staying on her current vertex.  The robber then moves similarly.  Players continue to alternate moves in this way.

Many variations on Cops and Robbers have been studied, the most common of which focus on altering the rules by which the players move.  For example, Aigner and Fromme \cite{AF84} and Neufeld and Nowakowski \cite{NN98} considered the so-called ``active'' version of the game, in which the robber must move on every robber turn -- that is, he cannot remain on his current vertex -- and at least one cop must move on every cop turn.  Recently, Offner and Ojakian \cite{OO14} introduced a wide class of Cops and Robbers variants, wherein one specifies how many cops must move on every cop turn, how many must remain in place, and how many may do either.  They focused in particular in the case where only one cop may move on each turn; this variant was termed \textit{Lazy Cops and Robbers} by Bal et al. \cite{BBKP15} and studied afterwards by several authors (see for example \cite{BBKP16,GY17,STW17}).

Inspired by the ``active'' game of Aigner and Fromme, we consider the variant of Cops and Robbers in which no player may ever remain on a vertex -- that is, every player must move on each turn.  We call this variant \textit{Fully Active Cops and Robbers}, and we refer to the original game as \textit{Passive Cops and Robbers}.  The minimum number of cops required for the cops to have a winning strategy in a graph $G$ is called the \textit{cop number} of $G$.  We use $\cop(G)$ to denote the cop number of $G$ in the passive setting and $\acop(G)$ to denote the cop number of $G$ in the fully active setting.

In this paper, we primarily focus on establishing values (or bounds on the value) of $\acop(G)$.  In Section \ref{sec:simple}, we give bounds on $\acop$ over several elementary classes of graphs; our main result of this section, Theorem \ref{thm:outerplanar}, states that $\acop(G) \le 2$ whenever $G$ is outerplanar.  In Section \ref{sec:comparison}, we investigate when and how much $\acop(G)$ and $\cop(G)$ can differ.  Theorem \ref{bounds} states that always $\cop(G)-1 \le \acop(G) \le 2\cop(G)$, while Theorem \ref{thm:bipartite_blowup} gives a class of graphs on which this upper bound on $\acop(G)$ is tight.  In Section \ref{sec:products}, we study the fully active game played on Cartesian products of graphs.  For general graphs, Theorem \ref{thm:prod_general} states that $\acop(G \cart H) \le \acop(G) \times \acop(H)$ provided that $G$ and $H$ are not both bipartite.  Theorem \ref{thm:cartesian_trees} gives the precise value of the fully active cop number of the Cartesian product of nontrivial trees, while Theorems \ref{thm:prod_cycles_nonbipartite} and \ref{thm:prod_cycles_bipartite} give bounds on the fully active cop number of a product of cycles.  We conclude with several tantalizing open questions.

\section{Some simple graph classes}\label{sec:simple}

Many of the strategies known for catching a robber in the passive game seem to fail in the fully active game.  For example, Aigner and Fromme showed \cite{AF84} that if $P$ is a shortest $uv$-path in $G$ (where $u$ and $v$ are any two distinct vertices in $G$), then one cop can always move along $P$ so that the robber may never occupy a vertex of $P$; this was the foundation of their proof that $\cop(G) \le 3$ whenever $G$ is planar.  However, this strategy relies on the ability of the cop to stay put if needed, and as such, it cannot be applied in the fully active game.  Thus, when studying the fully active game, we are forced to return to basics.

We will begin with some elementary results on Fully Active Cops and Robbers.  In the case of some simple, classic classes of graphs, it is easy to see that the cop number remains unchanged in the fully active setting.

\begin{prop}

    \begin{enumerate}
        \item If $T$ is a tree, then $\acop(T) = \cop(T) = 1$.
        \item If $n$ is a positive integer, then $\acop(C_n) = \cop(C_n) = 2$.
        \item If $n$ is a positive integer, then $\acop(K_n) = \cop(K_n) = 1$.
        \item If $m,n$ are positive integers, then $\acop(K_{m,n}) = \cop(K_{m,n}) = 2$.
    \end{enumerate}
\end{prop}

A graph is called \textit{cop-win} if one cop has a winning strategy (and \textit{robber-win} otherwise).  It is not hard to see that cop-win graphs in the passive setting are also cop-win in the active game.

\begin{prop}\label{copwin}
Let $G$ be a graph.  If $\cop(G) = 1$, then $\acop(G) = 1$.
\end{prop}

\begin{proof}
Suppose, to the contrary, that $G$ is robber-win in the fully active setting.  The robber could then play as follows in the passive setting.  When the cop moves, the robber plays his corresponding move from his winning strategy in the fully active setting.  When the cop remains in place, so does the robber.  This yields a winning strategy for the robber in the passive setting, a contradiction.
\end{proof}

We now provide a simple construction showing that the converse of Proposition \ref{copwin} does not hold and so, in general, $\acop(G)$ need not equal $\cop(G)$.

\begin{prop}\label{vAB}
Let $G$ be a graph with vertex partition $\{v\} \cup A \cup B$ and the following edges:
    \begin{itemize}
        \item $v$ is adjacent to all vertices in $A$,   % I'm not familiar with this terminology -- is it standard?  - Bill
        \item each vertex in $A$ is adjacent to all vertices in $B$, and
        \item $B$ is a clique.
    \end{itemize}
If $|A| \geq 2$ and $\size{B} \ge 2$ then $\acop(G) = 1$ and $\cop(G) = 2$.
\end{prop}
% Does this need some explanation?  And should it be |B| \ge 2 (also)?  Here is a proof if we want to include it.
\begin{proof}
In the passive game, the robber can evade one cop as follows.  If the cop begins the game on $v$, then the robber begins on any vertex in $B$; if the cop begins on some vertex in $A$, then the robber begins on any other vertex in $A$; if the cop begins on some vertex in $B$, then the robber begins on $v$.  Likewise, when the cop moves to $v$, the robber moves to some vertex in $B$; when the cop moves to a vertex in $A$, the robber moves to a different vertex in $A$; when the cop moves to a vertex in $B$, the robber moves to $v$.  Hence $\cop(G) \ge 2$, and it is easy to check that $\cop(G) \le 2$.  

In the fully active setting, a single cop can win as follows.  The cop begins on some vertex in $B$.  If the robber begins in $A \cup B$, then the cop can capture him immediately.  If instead the robber begins on $v$, then the cop moves to a different vertex of $B$.  The robber must now move to $A$, after which the cop can capture him.
\end{proof}

As mentioned earlier, a single cop cannot, in general, guard a shortest path in the fully active setting.  This strategy was crucial to Aigner and Fromme's proof \cite{AF84} that $\cop(G) \leq 3$ for any planar graph $G$, and as such, their proof cannot be adapted to the fully active setting.  It appears that determining a tight upper bound on $\acop$ for planar graphs requires some new ideas and may be difficult.  However, it is much easier to analyze the fully active game on outerplanar graph.  Clarke \cite{C02} showed that two cops can capture a robber on any outerplanar graph in the passive game; we show that the same is true in the fully active game.  The proof presented below is an adaptation of the proof from \cite{BN} of Clarke's result.

\begin{thm}\label{thm:outerplanar}
If $G$ is an outerplanar graph, then $\acop(G) \leq 2$.
\end{thm}

\begin{proof}
Let $G$ be an outerplanar graph.  Let us first assume that $G$ has no cut vertices.  If $G$ is a cycle, then clearly $\acop(G) \le 2$.  If not, then by definition $G$ may be embedded in the plane so that the outer face is bounded by a Hamiltonian cycle $C$ and all chords lie in the interior of $C$ and are non-crossing.  Let $C = v_1v_2\cdots v_nv_1$ and let $C_1$ and $C_2$ denote the two cops.  We say that a cop \textit{controls} a vertex $v$ if that cop's position is in the closed neighbourhood of $v$.  For two vertices $v_k$ and $v_{\ell}$, we define $v_kCv_{\ell}$ to be the path $v_kv_{k+1}\cdots v_{\ell-1}v_{\ell}$ (subscripts taken modulo $n$).

Throughout the game, cop $C_1$ will control some vertex $v_k$ and $C_2$ will control some vertex $v_{\ell}$.  The robber's position must lie on either $v_kCv_{\ell}$ or $v_{\ell}Cv_k$; call the interior of this path the \textit{robber territory} and the interior of the other path the \textit{cop territory}, and call $v_k$ and $v_{\ell}$ the \textit{endpoints} of the cop territory.  We show how the cops can move so that (a) after some finite number of moves the size of the cop territory has increased and (b) at no point does any edge join the robber's position to the cop territory.% (since no chord of $G$ can cross the chord occupied by the cops). BEN: omitted this since the cops do not control the ends of a chord in the final 2-connected case.
%Suppose that $C_1$ controls $v_k$ and $C_2$ controls  $v_{\ell}$.   

The cops begin by placing themselves at the ends of some chord and choosing $v_k$ and $v_{\ell}$ to be the vertices occupied by cops $C_1$ and $C_2$, respectively.  This ensures that no matter where the robber starts, no edge joins the robber's position to the cop territory.  %Now suppose that, at an aribtrary stage of the game, $C_1$ controls some vertex $v_k$ and $C_2$ controls $v_{\ell}$.% (note that $C_2$ must occupy one end of an edge $e$ which is incident to $v_{\ell}$)

Now consider an arbitrary point during the game just before the cops' turn, and suppose that cop $C_1$ controls $v_k$ while $C_2$ controls $v_{\ell}$; note that $C_2$ must occupy an endpoint of some edge, say $e$, incident to $v_{\ell}$.  Furthermore, suppose without loss of generality that the robber is on $v_kCv_{\ell}$ and suppose that no edge joins the robber's current position to the cop territory.  If some chord joins $v_k$ to some vertex in the robber territory, then let $v_r$ be the neighbour of $v_k$ in the robber territory which is closest along $C$ to $v_{\ell}$ (but not equal to $v_{\ell}$).  We may assume that $C_1$ occupies $v_k$, because if not she may move to $v_k$ while $C_2$ moves back and forth along $e$.  %, and any subsequent robber move keeps him in the robber territory.
If the robber's position is adjacent to $v_k$ then the cops clearly win, so suppose otherwise.  If the robber's position is on $v_rCv_{\ell}$, then $C_1$ moves to $v_r$ and $C_2$ moves along $e$.  The cops now control $v_r$ and $v_{\ell}$.  Note that by choice of $v_r$, no edge can join the robber's current position to the interior of $v_{\ell}Cv_r$: by assumption no edge joined the robber's current position to the interior of $v_{\ell}Cv_k$, the robber's position was not adjacent to $v_k$, and any edge joining the robber's position to the interior of $v_kCv_r$ would have to cross the chord $v_kv_r$.  The cops now set $v_k = v_r$, thereby enlarging the cop territory to $v_rCv_{\ell}$.  If the robber's position is in $v_kCv_r$, then $C_2$ moves to $v_r$ while $C_1$ moves along $v_kv_r$ as many times as necessary.  The cops then set $v_{\ell} = v_r$, enlarging the cop territory to $v_kCv_r$.  (Note that no edge can join the robber's current position to the cop territory, since any such edge would cross chord $v_kv_r$.)  A symmetric argument shows that the cops can enlarge the cop territory if $v_{\ell}$ has a chord to the robber territory but $v_k$ does not.  Finally, suppose that neither $v_k$ nor $v_{\ell}$ has a chord to the robber territory.  In this case, every path from the robber to the cop territory passes through either  edge $v_kv_{k+1}$ or edge $v_{\ell}v_{\ell-1}$.  In one step, $C_1$ can move to control $v_{k+1}$ and $C_2$ can move to control $v_{\ell-1}$.  The cops can then set $v_k = v_{k+1}$ and $v_{\ell} = v_{\ell-1}$, thereby increasing the cop territory to $v_{\ell-1}Cv_{k+1}$.

Now, suppose that $G$ has cut vertices and that the two cops occupy vertices in the same $2$-connected block $B$ of $G$.  If the robber's position is also in $B$, then the cops play as above.  If the robber's position is not in $B$, then there is some cut vertex $v$ which separates the cops' positions from the robber's position.  In this case, the cops play as if the robber occupies $v$.  In this way, the cops will either eventually capture the robber in $B$ or will both control $v$.  In the latter case, the cops may now move to a new $2$-connected block $B'$ where $V(B') \cap V(B) = \{v\}$ and repeat the strategy.  Since $v$ will always belong to the cops' territory in $B'$, the robber can never move back to $B$.  Eventually, the robber will be forced to an end-block of $G$ and will be caught.% by the strategy in the previous paragraph.
\end{proof}

\section{Passive versus fully active cop numbers}\label{sec:comparison}

In this section, we explore the relationship between $\acop(G)$ and $\cop(G)$.  We begin with elementary lower and upper bounds on $\acop(G)$ in terms of $\cop(G)$.

\begin{thm}\label{bounds}
If $G$ is a graph with $\cop(G) = k$, then $k-1 \leq \acop(G) \leq 2k$.
\end{thm}

\begin{proof}
We first prove the upper bound.  Let $\{C_1, \ldots, C_k, D_1, \ldots, D_k\}$ be a set of cops.  We place $\{C_1, \ldots, C_k\}$ on $V(G)$ according to a winning strategy in the passive setting and place each $D_i$ on some vertex adjacent to $C_i$.  The cops then play a modification of their winning strategy from the passive game.  If the strategy for the passive game requires $C_i$ to move to a new vertex, then she does so and the corresponding cop $D_i$ moves to the vertex formerly occupied by $C_i$.  If $C_i$ stays on her vertex in the passive game strategy, then $C_i$ and $D_i$ switch places and labels.  This ensures that at all times, all vertices occupied by cops in the passive game are also occupied by cops in the fully active game; since the cops eventually capture the robber in the passive game, they do so in the fully active game as well.

To prove the lower bound, we suppose that $\acop(G) = t$ and show how $t+1$ cops can win the passive game.  Let $\{C_1, \ldots, C_t, C^*\}$ be a set of cops.  We place $\{C_1, \ldots, C_t\}$ on $V(G)$ according to a winning strategy in the fully active setting and place $C^*$ arbitrarily.  The cops then play a modification of their winning strategy in the fully active game.  If the robber moves, then each $C_i$ moves according to the winning strategy and $C^*$ moves to decrease her distance from the robber.  If the robber remains in place, then so does each $C_i$, while $C^*$ again moves to decrease her distance from the robber.  Eventually, the game will be in a state where either the robber has been caught by some $C_i$ or $C^*$ occupies a vertex adjacent to the robber's position.  In the latter case, the robber can no longer remain in place, so the game proceeds as if it were being played in the fully active setting.  Thus, $\cop(G) \leq \acop(G) + 1$.
\end{proof}

We have seen that there exist graphs for which $\acop(G) = \cop(G)-1$ (Theorem \ref{vAB}), so the lower bound in Theorem \ref{bounds} is tight.  We next show that the upper bound is tight as well by producing, for each positive integer $k$, a class of  graphs $G$ having $\cop(G) = k$ and $\acop(G) = 2k$.

For a graph $G$ and positive integer $t$, the {\em $t$-blowup} of $G$ is a new graph obtained by replacing each vertex $v$ in $G$ with an independent set $S_v$ of size $t$ and replacing each edge $uv$ in $G$ with a complete bipartite graph having partite sets $S_u$ and $S_v$.  We refer to the vertices in $S_v$ as {\em copies} of $v$, and we call $v$ the {\em shadow} of each vertex in $S_v$.  We denote the $t$-blowup of $G$ by $G^{(t)}$.

We will need the following lemma, which is a special case of a result established by Schr\"{o}der (\cite{Sch14}, Theorem 2.7).  Note that the requirement that $\cop(G) \ge 2$ cannot be lifted: for $t \ge 2$, a robber may evade a single cop by always occupying a vertex that is distinct from the cop's vertex, yet has the same shadow.  

% Before giving our construction of graphs realizing the upper bound in Theorem \ref{bounds}, we establish a result on the passive game that we will use later.

\begin{lem}\label{lem:passive_blowup}
Let $G$ be a connected graph.  If $\cop(G) \ge 2$, then $\cop(G^{(t)}) = \cop(G)$ for all positive integers $t$.
\end{lem}

We are now ready to give our construction of a class of graphs having passive cop number $k$ and fully active cop number $2k$.  In the theorem below, we use $G \cart H$ to denote the Cartesian product of $G$ and $H$.

\begin{thm}\label{thm:bipartite_blowup}
Fix $k \ge 2$, let $T_1, T_2, \ldots, T_{2k-1}$ be nontrivial trees, and let $G = T_1 \cart T_2 \cart \cdots \cart T_{2k-1}$.  Now $\cop(G^{(t)}) = k$ and $\acop(G^{(t)}) = 2k$ whenever $t \ge 2k$.
\end{thm}
\begin{proof}
It follows from a result of Maamoun and Meyniel (\cite{MM87}, Theorem 2) that $\cop(G) = k$, hence Lemma~\ref{lem:passive_blowup} implies that $\cop(G^{(t)}) = k$.

For the active game, Theorem~\ref{bounds} shows that $\acop(G^{(t)}) \le 2\cop(G^{(t)}) = 2k$, so it suffices to show that $\acop(G^{(t)}) > 2k-1$.  Suppose the robber plays against $2k-1$ cops on $G^{(t)}$.  We view each vertex $v$ in $G$ as a $(2k-1)$-tuple $(v_1, v_2, \ldots, v_{2k-1})$, where $v_i \in V(T_i)$ for $1 \le i \le 2k-1$; we call $v_i$ the {\em $i$th coordinate} of $v$.  Note that every two adjacent vertices in $G$ differ in exactly one coordinate, and any two vertices at distance 2 differ in at most two coordinates.  Since each $T_i$ is bipartite, so are $G$ and $G^{(t)}$; let $X$ and $Y$ denote the partite sets of $G^{(t)}$. 

Once the cops have chosen their initial positions, some partite set, without loss of generality $X$, must contain at most $k-1$ cops.  For his initial position, the robber chooses any vertex in $Y$ that neither contains a cop nor is adjacent to a cop.  To see that this is possible, first note that for $v \in V(G^{(t)})$, we can bound the degree of $v$ by
$$\deg(v) \le t \cdot \sum_{i=1}^{2k-1} (\size{V(T_i)}-1).$$
Since at most $k-1$ cops occupy vertices of $X$ and at most $2k-1$ occupy vertices of $Y$, the number of vertices of $Y$ containing or adjacent to cops is at most $2k-1 + (k-1)t \cdot \sum_{i=1}^{2k-1}(\size{V(T_i)}-1)$.  We next estimate $\size{Y}$.  By symmetry, we may suppose that among $T_1, T_2, \ldots, T_{2k-1}$, none has fewer vertices than $T_{2k-1}$.  Suppose $u \in V(G^{(t)})$, let $v$ be the shadow of $u$ in $G$, and let $w$ be any neighbor of $v$ that agrees with $v$ in the first $2k-2$ coordinates.  Either $u \in Y$ and so all copies of $v$ belong to $Y$, or $u \in X$ and hence all copies of $w$ belong to $Y$.  Consequently, for all $v_1, v_2, \ldots, v_{2k-2}$ with $v_i \in V(T_i)$ for $1 \le i \le 2k-2$, there exists some $v_{2k-1}$ in $T_{2k-1}$ such that all copies of $(v_1, v_2, \ldots, v_{2k-1})$ belong to $Y$.  We conclude that
% the below isn't explained well (or at all, really), but should be true
$$\size{Y} \ge t \cdot \prod_{i=1}^{2k-2}\size{V_i} > 2k-1 + (k-1)t \cdot \sum_{i=1}^{2k-1} (\size{V(T_i)}-1),$$
so the robber may always choose a starting vertex that meets his criteria.

The robber's choice of initial position ensures that he cannot lose on the cops' first turn.  To show that the robber can avoid losing on subsequent turns, it suffices to show that he can always move to some vertex not containing a cop and not adjacent to a cop.  Suppose it is the robber's turn, and assume without loss of generality that the robber occupies some vertex in $Y$; a similar argument works for the other case.  Since $G^{(t)}$ is bipartite and the robber started in $Y$, he must have taken an even number of turns; consequently, the cops have taken an odd number of turns.  Thus every cop who started in $X$ is now in $Y$, and vice-versa.  In particular, there are at most $k-1$ cops in $Y$, and hence at most $k-1$ cops at distance 2 from the robber.  Let $v$ be the shadow in $G$ of the robber's current position, and let $w_1, w_2, \ldots, w_{2k-1}$ be the shadows of the cops' positions.  Since the robber was not adjacent to a cop after his previous turn, $v$ does not coincide with any of the $w_i$, and moreover, at most $k-1$ of the $w_i$ are at distance 2 from $v$.  Each $w_i$ at distance 2 from $v$ differs from $v$ in at most two coordinates, so there is some coordinate in which $v$ agrees with all such $w_i$.  Let $v'$ be some neighbor of $v$ that differs only in this coordinate, and note that $v'$ is not adjacent to any $w_i$ (although it might coincide with one or more).  Since $t \ge 2k$, there are at least $2k$ copies of $v'$.  As there are only $2k-1$ cops, some copy of $v'$ contains no cop.  The robber moves to any such copy of $v'$; by construction, the robber's new position contains no cop and is not adjacent to any cop, as desired.  It follows that the robber may evade capture indefinitely.
\end{proof}

Theorem~\ref{thm:bipartite_blowup} shows that for all $k \ge 2$, there exist graphs with cop number $k$ and fully active cop number $2k$, and thus $\acop(G) - \cop(G)$ can be arbitrarily large.  Note that Proposition \ref{copwin} shows that the requirement of $k \geq 2$ cannot be dropped from the previous statement.  
%In contrast, it is not hard to see that every graph with cop number 1 also has fully active cop number 1.  In the usual model of Cops and Robbers played with a single cop, it never helps the cop to remain on her current vertex, since the robber could respond by staying on his current vertex.  Hence, on any graph with cop number 1, there is a winning cop strategy on which the cop moves on every turn; this is a winning strategy in the fully active game as well.  

\section{Graph products}\label{sec:products}

In this section, we further consider the fully active game played on graph products.   For general graphs, we have the following result.  Note that the restriction imposed on $G_1$ and $G_2$ is not as stringent as it might first appear: for example, any non-bipartite graph meets this requirement, since in a non-bipartite graph the cops can reach any configuration from any other.  The restriction also holds for any graph with fully active cop number 1.

\begin{thm}\label{thm:prod_general}
If $G_1$ and $G_2$ are graphs such that $\acop(G_2)$ cops can win the fully active game on $G_2$ regardless of their initial positions, then $\acop(G_1 \cart G_2) \le \acop(G_1) + \acop(G_2)$. 
\end{thm}
\begin{proof}
First suppose that both $G_1$ and $G_2$ have the property that $\acop(G_i)$ cops can win the fully active game on $G_i$ regardless of their initial positions; we explain at the end of the proof how we can eliminate this restriction on $G_1$.  Let $k = \acop(G_1)$ and $\ell = \acop(G_2)$; we show that $k+\ell$ cops can win the fully active game on $G_1 \cart G_2$.  As usual, when a player occupies the vertex $(u,v)$ in $G_1 \cart G_2$, we say that $u$ (resp. $v$) is that player's {\em position in $G_1$} (resp. {\em position in $G_2$}).  

The cops will divide themselves into two teams, ``Team $G_1$'' and ``Team $G_2$''.  Initially, both teams are empty.  The cops fix winning strategies for $k+\ell$ cops in the fully active games on both $G_1$ and $G_2$.  Cop $i$ begins on vertex $(u,v)$, where $u$ denotes her starting position in the game on $G_1$ and $v$ denotes her starting position on $G_2$.  The cops now play as follows.  On their first turn, each cop moves in $G_1$ according to the cops' winning strategy for the game on $G_1$.  On subsequent turns, the cops respond to the robber's previous move as follows: if the robber moved in $G_i$, then the cops move in $G_i$ according to a winning strategy for the game on $G_i$.  Eventually, some cop must capture the robber in one of the games.  If a cop captures the robber in the game on $G_i$, then her position in $G_i$ agrees with the robber's, and the cop joins Team $G_i$.  Henceforth, whenever the robber moves in $G_i$, each cop in Team $G_i$ makes an identical move in $G_i$ (and when the robber moves in the other graph, these cops make arbitrary moves in that graph).  This ensures that each cop in Team $G_i$ always occupies the same position as the robber in $G_i$.

The cops now repeat this process: the remaining $k+\ell-1$ cops mimic winning strategies in the fully active games on $G_1$ and $G_2$ until one cop captures the robber either on $G_1$ (in which case she joins Team $G_1$) or on $G_2$ (in which case she joins Team $G_2$).  The cops continue in this manner.   Eventually, either $\ell$ cops have joined Team $G_1$ or $k$ cops have joined Team $G_2$; without loss of generality, assume the former.  

The cops now change their strategies slightly.  Henceforth, when the robber moves in $G_1$, the cops in Team $G_1$ continue to follow him; when the robber moves in $G_2$, these cops move in $G_2$ according to a winning strategy for the game on $G_2$.  Note that the robber cannot move in $G_2$ infinitely often, or some cop in Team $G_1$ will reach the same position as the robber in $G_2$, at which point she captures the robber in $G_1 \cart G_2$.  

The remaining cops play slightly differently.  When the robber moves in $G_1$, the cops in Team $G_2$ move arbitrarily in $G_1$, while the unassigned cops move toward the robber in $G_2$.  When the robber moves in $G_2$, the cops in Team $G_2$ follow him in $G_2$, while the unassigned cops again move toward him in $G_2$.  If at any point one of the unassigned cops reaches the same position as the robber in $G_2$, then that cop joins Team $G_2$.  Since the robber cannot move in $G_2$ infinitely often, he must move in $G_1$ infinitely often, so eventually all of the unassigned cops catch up to the robber in $G_2$ and thus join Team $G_2$.  Thus eventually Team $G_2$ comes to contain $k$ cops.  At this point, whenever the robber moves in $G_1$, the cops in Team $G_2$ move through $G_1$ according to a winning strategy in the game on $G_1$.  (When he moves in $G_2$, they continue to follow him in $G_2$.)  Since the robber moves in $G_1$ infinitely often, eventually one of the cops in Team $G_2$ will capture the robber in $G_1$ and hence in $G_1 \cart G_2$.

This completes the proof when both $G_1$ and $G_2$ have the property that $\acop(G_i)$ cops can capture the robber on $G_i$ regardless of their starting positions.  Suppose now that $G_2$ has this property but $G_1$ does not.  Note that $G_1$ must necessarily be bipartite; call its partite sets $X$ and $Y$.  Consider a winning strategy for the game on $G_1$, and suppose that under this strategy $k_X$ cops begin on vertices of $X$ while $k_Y$ begin on vertices of $Y$.  We claim that $k_X+k_Y$ cops can win the game on $G_1$ for any initial arrangement that places $k_X$ cops in $X$ and $k_Y$ in $Y$: indeed, from any such arrangement, the cops can reconfigure themselves to their desired starting positions.

The cops amend their strategy on $G_1 \cart G_2$ as follows.  All cops assigned to Team $G_1$ or Team $G_2$ play as normal, but the unassigned cops play slightly differently.  After any cop joins a team, before restarting their strategy to capture the robber on $G_2$, the unassigned cops slightly rearrange their positions.  If fewer than $k_X$ cops in Team $G_2$ currently occupy vertices whose position in $G_1$ belongs to $X$, then the unassigned cops move so that, after an even number of steps, every unassigned cop's position in $G_1$ belongs to $X$.  Should one of these cops join Team $G_2$ on the next iteration of the strategy, then after an even number of steps, her position in $G_1$ will belong to $X$.  If instead exactly $k_X$ cops in Team $G_2$ occupy vertices whose position in $G_1$ belongs to $X$, then the unassigned cops move so that after an even number of steps, each occupies a vertices whose position in $G_1$ belongs to $Y$; this ensures that if one of these cops joins Team $G_2$, then her position in $G_1$ will belong to $Y$.  By playing thus, the cops ensure that once Team $G_2$ has been completely filled, the cops' positions in $G_1$ will correspond to an initial configuration from which they can win the game on $G_1$.  From this point, the cops resume following the strategy above, which ensures that they eventually capture the robber. 
\end{proof}

Equality in Theorem~\ref{thm:prod_general} does not hold in general.  For example, $\acop(K_2) = 1$ and $\acop(C_4) = 2$, but $\acop(K_2 \cart C_4) = 2 < \acop(K_2) + \acop(C_4)$.  Note also that the restriction on $G_2$ cannot be lifted.  Fix any positive integers $k,t$ with $k \ge 2$ and $t \ge 4k$.  By Theorem \ref{thm:bipartite_blowup}, we have $\acop(Q_{2k-1}^{(t)}) = 2k$.  By Theorem \ref{OO} (see below), we have $\acop(Q_{2k}) = \lceil 4k/3 \rceil$.  But now
$$\acop(Q_{2k-1}^{(t)} \cart Q_{2k}) = \acop((Q_{2k-1}\cart Q_{2k})^{(t)}) = \acop(Q_{4k-1}^{(t)}) = 4k$$
by Theorem \ref{thm:bipartite_blowup}, but 
$$\acop(Q_{2k-1}^{(t)}) + \acop(Q_{2k}) = 2k + \left\lceil \frac{4k}{3} \right\rceil.$$

We now turn our attention to more restricted classes of Cartesian products.  Possibly the most notable of all graph products is the $n$-dimensional hypercube, denoted $Q_n$.  Offner and Ojakian \cite{OO14} studied a wide class of variants of Cops and Robbers played on the hypercube, in which some cops must move on every cop turn, while others have the option of remaining in place.  The theorem below is a special case of their results.

\begin{thm}\label{OO}
For any positive integer $n$, we have $\acop(Q_n) = \left\lceil\dfrac{2n}{3}\right\rceil.$
\end{thm}

% Since $\cop(Q_n) = \left\lceil\dfrac{n+1}{2}\right\rceil$ (see \cite{MM87}), we thus have the following:

%\begin{cor}
%For any positive integer $k$, there exists a graph for which $\size{\acop(G) - \cop(G)} > k$.
%\end{cor}

We extend Theorem~\ref{OO} to the more general setting where the game is played on the Cartesian product of arbitrary nontrivial trees.  We begin with a lemma.

\begin{lem}\label{lem:tree_odd}
Let $T_1$ and $T_2$ be nontrivial trees, and consider the fully active game played with a single cop on $T_1 \cart T_2$.  If the cop and robber begin the game at odd distance from each other, then the cop can capture the robber.
\end{lem}
\begin{proof}
We view vertices of $T_1 \cart T_2$ as pairs $(v_1,v_2)$ with $v_i \in V(T_i)$; when a player is located at this vertex, we call $v_i$ that player's {\em position in $T_i$}.  We give a winning strategy for the cop.  The cop's strategy is simple.  At each point in the game, let $d_i$ denote the distance (in $T_i$) from the cop's position in $T_i$ to the robber's position in $T_i$.  On each cop turn, if $d_1 > d_2$, then the cop takes one step closer to the robber in $T_1$; otherwise, she takes one step closer to the robber in $T_2$.  Note that since $T_1 \cart T_2$ is bipartite and the cop is at an odd distance from the robber before her first turn, she must be at odd distance from the robber before each of her turns for the duration of the game.  Hence $d_1 \not = d_2$ and, consequently, the cop's move always decreases $\max\{d_1,d_2\}$ by 1.

To see that the cop eventually captures the robber, it suffices to show that $\max\{d_1, d_2\}$ gradually decreases throughout the game.  It is clear that $\max\{d_1, d_2\}$ never increases over the course of a full round (that is, a cop turn together with the subsequent robber turn): the cop's move decreases $\max\{d_1,d_2\}$ by 1, while the robber's move increases it by at most 1.  Moreover, the robber can increase $\max\{d_1,d_2\}$ only by moving away from the cop in the appropriate tree.  However, he cannot do this forever: the robber can take at most $\diam(T_1)$ steps away from the cop in $T_1$ and at most $\diam(T_2)$ steps in $T_2$, so after at most $\diam(T_1)+\diam(T_2)$ rounds the robber must take at least one step toward the cop.  Thus, in this round $\max\{d_1,d_2\}$ decreases.  Consequently, $\max\{d_1,d_2\}$ eventually reaches 0, at which point the cop has captured the robber.
\end{proof}

\begin{thm}\label{thm:cartesian_trees}
Let $\{T_{1}, T_{2}, \ldots, T_{k}\}$ be nontrivial trees. If $G = T_{1} \cart T_{2} \cart \cdots \cart T_{k}$, then $\acop(G) = \left\lceil\dfrac{2k}{3}\right\rceil$.
\end{thm}

\begin{proof}
The claim is clearly true when $k=1$, so suppose $k \ge 2$.  We first show that $\left\lceil 2k/3\right\rceil$ cops can capture the robber.  Initially, the cops split themselves into groups whose sizes differ by at most 1.  The cops in one group choose some vertex $v$ and all begin the game on $v$; cops in the other group begin on any neighbor of $v$.  Once the robber has chosen his initial position, let $c$ denote the number of cops at even distance from the robber (call them {\em even cops}) and $d$ the number of cops at odd distance ({\em odd cops}).  Note that since $G$ is bipartite, there will be $c$ even cops and $d$ odd cops prior to every cop turn throughout the duration of the game.

We assign to each cop either one or two {\em inactive coordinates}; coordinates which are not inactive are {\em active coordinates}.  To the $c$ even cops we assign coordinates $1, 2, \dots, c$, respectively; to the $d$ odd cops we assign coordinates $c+1 \text{ and } c+2, c+3 \text{ and } c+4, \dots, c+2d-1 \text{ and } c+2d$, respectively.  If necessary, we ``round down'' active coordinates to $k$, so it may be that multiple cops have $k$ as an inactive coordinate.  Through case analysis depending on the congruence class of $k$ modulo 3, it is easily verified that $c+2d \ge k$, so every coordinate is an inactive coordinate for at least one cop.

Each cop moves as follows.  If the cop's and robber's positions disagree in any of the cop's active coordinates, then the cop takes one step closer to the robber in any such coordinate.  Otherwise, the cop restricts her attention to her inactive coordinates and pretends she is playing a game on the Cartesian product of the corresponding trees.  If the cop has only one inactive coordinate, then she simply moves one step closer to the robber in that coordinate.  If instead she has two inactive coordinates, then she follows the strategy outlined in Lemma~\ref{lem:tree_odd}.  Note that every cop with two inactive coordinates is an odd cop, and hence must be at an odd distance from the robber; since the cop under consideration agrees with the robber in all but her two inactive coordinates, she must be at odd distance even when considering only those two coordinates.  Thus, she can indeed follow the winning strategy in Lemma~\ref{lem:tree_odd}.  Once the cop captures the robber in this new game, she has in fact captured the robber in the ``real'' game.

It remains to show that this is a winning strategy for the cops.  Consider a single round of the game, consisting of a robber turn followed by a cop turn.  It is clear from the cops' strategy that each cop's total distance to the robber across all active coordinates cannot increase throughout the course of a full round.  Suppose now that the robber moved in coordinate $i$.  Some cop has $i$ as an inactive coordinate.  If, on the cops' turn, that cop had not yet caught up to the robber in all of her active coordinates, then by the end of the round her total distance to the robber across all active coordinates has decreased by one.  If instead that cop had already caught up to the robber in her active coordinates, then on her turn, she was able to focus on her inactive coordinates and take one step closer to winning in that game.  Thus, on each turn, at least one cop makes progress, either toward catching up to the robber in her active coordinates or toward capturing the robber in her inactive coordinates; moreover, no other cop loses progress toward either of these goals.  Hence, eventually, some cop captures the robber.

We have thus shown that $\acop(G) \le \left\lceil 2k/3\right\rceil$.  To show that $\acop(G) \ge \left\lceil 2k/3\right\rceil$, we give a strategy for the robber to evade $\left\lceil 2k/3\right\rceil - 1$ cops.  Denote the partite sets of $G$ by $X$ and $Y$, and suppose that initially $c$ cops begin in $X$ while $d$ cops begin in $Y$.  Without loss of generality suppose $c \le d$.  

The robber begins the game on any vertex in $Y$ that is neither occupied by nor adjacent to any cop.  An argument similar to that used in the proof of Theorem~\ref{thm:bipartite_blowup} shows that this is always possible.  It suffices to show that, on every robber turn, the robber can move to a vertex that is neither occupied by nor adjacent to any cop.  Since cops at distance three or greater from the robber pose no immediate threat to him, we focus only on cops that are either at distance 1 or distance 2.  When it is the robber's turn, the cops have taken one more turn than the robber, so $c$ cops occupy the same partite set as the robber while $d$ cops occupy the other partite set; consequently, at most $c$ cops are at distance 2, while at most $d$ are at distance 1.  Cops at distance 2 agree with the robber in all but at most two coordinates, while those at distance 1 agree in all but at most one coordinate.  Now note that
$$2c+d = c+(c+d) \le \frac{3}{2}\left (\left\lceil\dfrac{2k}{3}\right\rceil - 1\right ) < k,$$
so there is at least one coordinate, say coordinate $i$, in which all cops at distance 1 or 2 agree with the robber.  The robber now simply takes one step in any direction in $T_i$.  This increases his distance to those cops at distance 1 or 2, while it decreases his distance to all other cops by at most 1.  Thus the robber ends his turn on a vertex that is neither occupied by nor adjacent to any cop, as desired.
\end{proof}

We next seek to determine the fully active cop numbers of products of cycles.  We will need the following lemma.

\begin{lem}\label{lem:same_partite_set}
Let $G$ be a bipartite graph.  Let $k = \cop(G)$, and consider the fully active game on $G$ with $k$ cops.  If, after the initial placement by both players, all cops and the robber occupy the same partite set of $G$, then the cops can ensure capture of the robber.
\end{lem}
\begin{proof}
Label the cops $C_1, C_2, \ldots, C_k$.  Suppose that the cops and robber have chosen their initial positions for the fully active game on $G$, and suppose further that all cops and the robber occupy the same partite set.  The cops ``imagine'' an instance of the passive game on $G$ and use a winning strategy in that game to guide their play in the fully active game.  Initially, all cops and the robber occupy the same vertices in the imagined passive game as in the fully active game.  (Note that when $\cop(G)$ cops play the passive game on $G$, they can win regardless of which positions they initially occupy, since they can begin the game by gradually moving to whichever positions they might have preferred to start at.)  For positive integers $i, t$ with $1 \le i \le k$, let $v_i^{(t)}$ denote the position of cop $C_i$ in the passive game after $t$ cop turns.  

In the passive game, the cops follow a winning strategy.  They would like to employ the same strategy in the fully active game; however, cops in the passive game may choose to remain in place, while those in the fully active game cannot.  Hence the cops cannot ensure that each cop always occupies the same vertex in both games.  However, the cops can ensure that for all positive integers $t$, after $t$ cop turns, each cop $C_i$ occupies either vertex $v_i^{(t)}$ or one of its neighbors.  This is clearly true after the cops' initial placement (adopting the convention that $v_i^{(0)}$ denotes $C_i$'s initial position), so we need only show how the cops can maintain this invariant from round to round.  The robber will always occupy the same vertex in both games.

The cops play as follows.  Whenever the robber makes a move in the fully active game, the cops imagine that he makes the same move in the passive game, and they respond (in the passive game) as dictated by their winning strategy.  In the fully active game, the cops mimic this response in the following manner.  Suppose the cops have played $t$ turns in the passive game.  By the invariant, each cop $C_i$ occupies $v_i^{(t)}$ in the passive game and either $v_i^{(t-1)}$ or one of its neighbors in the fully active game.  If $C_i$ occupies some neighbor of $v_i^{(t-1)}$ in the fully active game, then he moves to $v_i^{(t-1)}$ itself; since $v_i^{(t-1)}$ and $v_i^{(t)}$ must be adjacent, this maintains the invariant.  If instead $C_i$ occupies $v_i^{(t-1)}$ itself and $v_i^{(t-1)} \not = v_i^{(t)}$, then he moves to $v_i^{(t)}$.  Finally, if $C_i$ occupies $v_i^{(t-1)}$ and $v_i^{(t-1)} = v_i^{(t)}$, then he moves to any neighbor of $v_i^{(t)}$.  In any case, the invariant is maintained.

Eventually, some cop $C_i$ captures the robber in the passive game.  When this happens, suppose $C_i$ and the robber occupy vertex $v$ in the passive game, while $C_i$ occupies vertex $u$ in the fully active game.  By the cops' invariant, either $u = v$ or $u \in N(v)$.  If the passive game capture happens on a robber turn, then regardless of whether $u = v$ or $u \in N(v)$, cop $C_i$ either has already captured the robber in the fully active game or can capture him on the ensuing cop turn.  If instead the passive game capture happens on a cop turn, then the cops and robber have made the same number of moves; since the cops all began on the same partite set as the robber, in the fully active game, the cops and robber must all occupy the same partite set.  In particular, $u$ and $v$ belong to the same partite set, so we cannot have $u \in N(v)$, and must therefore have $u = v$: that is, cop $C_i$ has in fact captured the robber in the fully active game.  In either case, the cops win the fully active game.
\end{proof}

Later in the section, we will apply Lemma~\ref{lem:same_partite_set} to a special type of graph -- a {\em covering graph}.

\begin{defn}
Given graphs $G$ and $H$, a {\em covering map} from $H$ onto $G$ is a mapping from $V(H)$ to $V(G)$ that is surjective and locally isomorphic -- that is, for each vertex $v$ in $H$, the neighborhood of $v$ maps bijectively onto the neighborhood of its image in $G$.  When a covering map from $H$ to $G$ exists, we say that $H$ is a {\em covering graph} of $G$.
\end{defn}

\begin{lem}\label{lem:covering}
Let $G$ and $H$ be graphs, and let $\phi : H \rightarrow G$ be a covering map.  Consider a multiset ${\mathcal C}$ of cop positions and a robber position $R$ in $H$.  If cops who begin the game on ${\mathcal C}$ can capture a robber who begins on $R$ in the game on $H$, then cops who begin on $\phi({\mathcal C})$ can capture a robber who begins on $\phi(R)$ in the game on $G$.
\end{lem}
\begin{proof}
When playing the game on $G$, the cops ``imagine'' a game on $H$ and use a winning strategy in that game to guide their play on $G$.  In the game on $H$, the cops initially occupy the multiset ${\mathcal C}$ of positions, while the robber occupies position $R$; in the game on $G$, the cops occupy $\phi({\mathcal C})$, while the robber occupies $\phi(R)$.  The cops maintain the invariant that the position of each entity (that is, every cop and the robber) in the game on $H$ maps, under $\phi$, to the position of that entity in the game on $G$.  This is clearly true at the beginning of the game.  When the robber on $G$ moves from some vertex $u$ to some adjacent vertex $v$, the cops imagine that the robber on $H$ moves to some vertex $v'$ with $\phi(v') = v$; this is possible since, by the invariant, the robber currently occupies some vertex $u'$ with $u = \phi(u')$, and since $\phi$ is isomorphic on $N(u')$.  On the cops' turn in the game on $G$, each cop first moves on $H$ according to some winning strategy for that game, and then moves, in $G$, to the image of his new position; as before, this is possible because $\phi$ is locally isomorphic.  Since the cops play a winning strategy on $H$, eventually some cop on $H$ occupies the same vertex (say $x$) as the robber.  At this point, by the invariant, that cop and the robber both occupy $\phi(x)$ in the game on $G$, so the cops must eventually capture the robber. 
\end{proof}
%
%In what follows, we need the following easy observation:
%\begin{proposition}\label{prop:cycle_cover}
%For $k \ge 3$, the graph $C_{2k}$ is a covering graph of $C_k$; moreover, there is some covering map under which each vertex of $C_k$ has one preimage in each partite set of $C_{2k}$.
%\end{proposition}
%\begin{proof}
%Let $u_1, u_2, \ldots, u_{2k}$ and $v_1, v_2, \ldots, v_k$ denote the vertices of $C_{2k}$ and $C_k$, respectively, in cyclic order.  The map $\phi : \{u_1, u_2, \ldots, u_{2k}\} \rightarrow \{v_1, v_2, \ldots, v_k\}$ given by
%$$\phi(u_i) = \left \{\begin{array}[ll]v_i &\text{if } i \le k;\\v_{i-k} &\text{if }i > k\end{array} \right .$$
%is easily seen to be a covering map with the desired property. 
%\end{proof}

%\begin{proposition}\label{prop:product_cover}
%Fix graphs $G$, $G'$, and $H$.  If $G'$ is a covering graph of $G$, then $G' \cart H$ is a covering graph of $G \cart H$.
%\end{proposition}
%\begin{proof}
%Let $\phi$ be a covering map from $G'$ onto $G$.  The map $\psi : V(G' \cart H) \rightarrow V(G \cart H)$ given by $\psi((u,v)) = (\phi(u),v)$ for all $u \in V(G'), v \in V(H)$ is easily seen to be a covering map.
%\end{proof}

%\textbf{*** TODO: in the theorem below, can we reduce the upper bound by 1 if any of the factors is $C_3$ or $C_4$?  What about the lower bound?  The easy lower bound is $k$.  Can we improve it? ***} \\

We now have the tools we need to analyze the fully active game played on the Cartesian product of cycles.  In the theorem below, we make use of the fact that $\cop(C_{n_1} \cart C_{n_2} \cart \cdots \cart C_{n_k}) = k+1$ for any positive integers $k, n_1, n_2, \ldots, n_k$, a result due to Neufeld and Nowakowski \cite{NN98}.

\begin{thm}\label{thm:prod_cycles_nonbipartite}
Let $G = C_{n_1} \cart C_{n_2} \cart \cdots \cart C_{n_k}$.  If any of the $n_i$ is odd, then $\acop(G) \le k+1$.
\end{thm}
\begin{proof}
Suppose without loss of generality that $n_1$ is odd.  Let $H = C_{2n_1} \cart C_{2n_2} \cart \ldots \cart C_{2n_k}$.  We represent the vertices of $G$ by ordered $k$-tuples $(w_1, w_2, \ldots, w_k)$ with $0 \le u_i < n_i$ in the usual way: two vertices are adjacent if and only if they agree in all but one coordinate, where they differ by 1 (modulo the length of the corresponding factor cycle).  Likewise, we represent the vertices of $H$ by ordered $k$-tuples $(x_1, x_2, \ldots, x_k)$ with $0 \le x_i < n_{2i}$.  For $1 \le i \le k$, let $\phi_i$ be the covering map from $C_{2n_i}$ onto $C_{n_i}$ defined by 
$$\phi_i(x) = \left \{\begin{array}{ll}x &\text{if } x \le k;\\x-k &\text{if }x > k\end{array} \right .$$
It is easily verified that the map $\psi : V(H) \rightarrow V(G)$ defined by $\psi((v_1, \ldots, v_k)) = (\phi_1(v_1), \ldots, \phi_k(v_k))$ is a covering map from $H$ onto $G$.

Consider the fully active game on $G$, played with $k+1$ cops.  We show how the cops can use an ``imagined'' game on $H$ to guide them in playing on $G$.  In the game on $G$, the cops all begin at vertex $(0,0,\ldots,0)$.  Suppose the robber begins at $(r_1, r_2, \ldots, r_k)$.  Let $(r_1', r_2', \ldots, r_k')$ be any vertex of $H$ whose image (under $\psi$) is $(r_1, r_2, \ldots, r_k)$; the cops imagine that the robber begins the game on $H$ at vertex $(r_1', r_2', \ldots, r_k')$.  If $(0,0,\ldots,0)$ and $(r_1', r_2', \ldots, r_k')$ belong to the same partite set of $H$, then the cops imagine that they all begin the game on $H$ at $(0,0,\ldots, 0)$; otherwise, the cops imagine that they all begin at $(n_1, 0, 0, \ldots, 0)$.  In either case, in the game on $H$, the cops and robber all occupy the same partite set.  Since $\cop(H) = k+1$, Lemma~\ref{lem:same_partite_set} implies that the cops have a winning strategy on $H$.  Consequently, by Lemma~\ref{lem:covering}, they have a winning strategy on $G$ as well, so $\acop(G) \le k+1$, as claimed.
\end{proof}

Theorem~\ref{thm:prod_cycles_nonbipartite} and Theorem~\ref{bounds} together show that when $C_{n_1} \cart C_{n_2} \cart \dots \cart C_{n_k}$ is non-bipartite, then its fully active cop number is either $k$ or $k+1$.  As it turns out, when the graph is bipartite, the fully active cop number behaves much differently.  

\begin{thm}\label{thm:prod_cycles_bipartite}
Let $G = C_{n_1} \cart C_{n_2} \cart \cdots \cart C_{n_k}$.  If the $n_i$ are all even, then $\ceil{\frac{4k}{3}} \le \acop(G) \le \ceil{\frac{4k+4}{3}}$. 
\end{thm}
\begin{proof}
The lower bound follows from an argument very similar to that used in the proof of Theorem \ref{thm:cartesian_trees}; we omit the details.

For the upper bound, suppose each $n_i$ is even and let $m = \ceil{\frac{4k+4}{3}}$.  We give a strategy for $m$ cops to capture the robber on $G$.  We represent the vertices of $G$ as ordered $k$-tuples in the usual way.  Each player's move consists of either incrementing or decrementing one coordinate of his current position; when a player increments (resp. decrements) his $i$th coordinate, we call this {\em moving forward} (resp. {\em moving backward}) in dimension $i$.  We consider several cases, depending on the congruence class of $k$ modulo 3.  

{\bf Case 1:} $k = 3\ell$ for some integer $\ell$.  In this case, $m = 4\ell+2$.  Out of the $4\ell+2$ cops, $2\ell+1$ begin the game on vertex $(0,0, \ldots, 0)$, while the other $2\ell+1$ begin on $(1, 0, \ldots, 0)$.  No matter where the robber starts, exactly $2\ell+1$ cops begin in the same partite set as the robber, while $2\ell+1$ begin in the other partite set.  Let $S$ be the set of cops that begin in the same partite set as the robber, and partition $2\ell$ of the remaining cops into sets $T_1, T_2, \ldots, T_{\ell}$, each of size 2.  (The final remaining cop is not needed and may move arbitrarily.)  

The cops now employ the following strategy.  The cops in $S$ initially aim to make their positions agree with the robber's in dimensions $2\ell+1, 2\ell+2, \ldots, 3\ell$: they do this greedily, always moving closer to the robber in one of these dimensions.  After the cops have achieved this goal, they employ a different strategy.  Whenever the robber moves in dimensions $2\ell+1, 2\ell+2, \ldots, 3\ell$, the cops in $S$ mirror this action, moving in the same dimension and in the same direction.  When the robber instead moves in one of the first $2\ell$ dimensions, the cops in $S$ take one more step in a winning strategy in the projection of the game onto the first $2\ell$ dimensions.  (The existence of such a strategy is guaranteed by Lemma~\ref{lem:same_partite_set}.)  The cops in each $T_i$ employ a similar strategy.  First, they greedily attempt to catch up to the robber in all dimensions other than $2\ell+i$; once they have done so, whenever the robber moves in dimension $2\ell+i$ the cops in $T_i$ take one more step in a winning strategy in the projection of the game onto dimension $2\ell+i$, and whenever the robber moves in any other dimension the cops move in the same dimension and the same direction.  

Within the first $n_{2\ell+1} + n_{2\ell+2} + \ldots + n_{3\ell}$ times the robber moves in one of the first $2\ell$ dimensions, the cops in $S$ catch up to him in dimensions $2\ell+1, 2\ell+2, \ldots, 3\ell$; after finitely many more robber moves in one of the first $2\ell$ dimensions, some cop in $S$ captures the robber.  Thus, the robber can move in one of the first $2\ell$ dimensions only finitely many times without being captured.  Likewise, for $1 \le i \le \ell$, the robber can move only finitely many times in dimension $2\ell+i$ before he is captured by some cop in $T_i$.  In other words, whenever the robber moves in dimensions $1, 2, \ldots, 2\ell$, the cops in $S$ get closer to capturing him, and whenever the robber moves in dimension $2\ell+i$, the cops in $T_i$ get closer to capturing him. Eventually, some cop must capture the robber.

{\bf Case 2:} $k = 3\ell+1$ for some integer $\ell$.  We now have $m = 4\ell+3$.  We proceed similarly to the previous case, starting $2\ell+1$ cops on $(0,0,\ldots, 0)$ and the other $2\ell+2$ on $(1,0,\ldots, 0)$.  At least $2\ell+1$ cops must begin in the same partite set as the robber; let these cops comprise the set $S$, and partition the remaining $2\ell+2$ cops into pairs $T_1, T_2, \ldots, T_{\ell+1}$.  As in Case 1, the cops in $S$ first catch up to the robber in dimensions $2\ell+1, 2\ell+2, \ldots, 3\ell+1$, then attempt to employ a winning strategy in the projection of the game onto the first $2\ell$ dimensions; the cops in $T_i$ first catch up to the robber in all dimensions other than $T_{2\ell+i}$, then attempt to employ a winning strategy in the projection of the game onto dimension $2\ell+i$.  As before, some cop eventually captures the robber.

{\bf Case 3:} $k = 3\ell+2$ for some integer $\ell$.  This time, $m = 4\ell+4$.  Now $2\ell+2$ cops begin on $(0,0,\ldots, 0)$, while the other $2\ell+2$ begin on $(1,0,\ldots, 0)$.  Exactly $2\ell+2$ cops must begin in the same partite set as the robber; let these cops comprise the set $S$, and partition the remaining $2\ell+2$ cops into pairs $T_1, T_2, \ldots, T_{\ell+1}$.  As in the previous cases, the cops in $S$ first catch up to the robber in dimensions $2\ell+2, 2\ell+3, \ldots, 3\ell+2$, then attempt to employ a winning strategy in the projection of the game onto the first $2\ell+1$ dimensions; the cops in $T_i$ first catch up to the robber in all dimensions other than $T_{2\ell+1+i}$, then attempt to employ a winning strategy in the projection of the game onto dimension $2\ell+1+i$.  Once again, some cop eventually captures the robber.

In any case, $m$ cops suffice to capture a robber on $G$.
\end{proof}

\section{Open problems}

Several natural questions on the fully active game remain open.  

\begin{ques}
What is the smallest constant $c$ such that $\acop(G) \leq c$ for every planar graph $G$?
\end{ques}

Trivially, $\acop(G) \leq 6$ for every planar graph $G$, following from Aigner and Fromme's proof \cite{AF84} that planar graphs have cop number at most $3$ and from Theorem \ref{bounds}.  %However, since their proof relies on retracts, no easy modification of their argument seems possible to achieve an improved result for active cops and robbers.  Thus, any progress on Question \cite{planar} will most likely require a new analysis of planar graphs.

\begin{ques}
Is $\acop(G) \leq \cop(G)$ for every non-bipartite graph $G$?
\end{ques}

\begin{ques}
Is $\acop(G) \geq \cop(G)$ for every bipartite graph $G$?
\end{ques}

The only examples of graphs we know that satisfy $\acop(G) < \cop(G)$ are non-bipartite, and the only ones we know that satisfy $\acop(G) > \cop(G)$ are bipartite.  In light of Theorems \ref{thm:prod_general}, \ref{thm:prod_cycles_nonbipartite}, and \ref{thm:prod_cycles_bipartite}, it seems that the game works very differently on bipartite and non-bipartite graphs, so it is not unreasonable to suspect that $\acop(G) > \cop(G)$ requires $G$ to be bipartite.

\begin{ques}
Let $c$ and $k$ be positive integers, with $k<c$.  Does there necessarily exist a graph $G$ with $\cop(G)=c$ and $\acop(G)=c+k$?
\end{ques}

As demonstrated by Cartesian products of even cycles (Theorem~\ref{thm:prod_cycles_bipartite}) and blowups of Cartesian products of trees (Theorem~\ref{thm:bipartite_blowup}), we know that the active cop number or a graph can be roughly $4/3$ or $2$ times the usual cop number, but we have no construction that forces the parameters to differ by a prescribed additive constant.

\section{Acknowledgements}

This research was conducted, in part, while the first author was affiliated with Dawson College.  The first and third authors received support from the Fonds de recherche du Qu\'ebec - Nature et technologies.

\section*{\refname}

\bibliographystyle{siam}
\bibliography{references}

\end{document}